\documentclass[12pt]{amsart}
\usepackage{graphicx, amsthm, amsfonts}

\theoremstyle{plain}
\newtheorem{thm}{Theorem}[section]
\newtheorem{cor}[thm]{Corollary}
\newtheorem{lem}[thm]{Lemma}

\theoremstyle{definition}
\newtheorem{defn}[thm]{Definition}

\theoremstyle{remark}
\newtheorem{rem}[thm]{Remark}

\newcommand{\nn}{\ensuremath{\mathbb{N}}}
\newcommand{\p}{\mathbb P}
\newcommand{\e}{\mathbb E}
\newcommand{\lr}{\left(}
\newcommand{\rd}{\right)}
\newcommand{\lc}{\left\{}
\newcommand{\rc}{\right\}}
\begin{document}

\title{Bounds on the Maximum Number of Minimum Dominating Sets}

\author{Samuel Connolly, Zachary Gabor, Anant Godbole, Bill Kay}

\maketitle

\centerline{
University of Pennsylvania, Haverford College,}

\centerline{ East Tennessee State University, and Emory University }

% ********************
\begin{abstract}
We use probabilistic methods to find lower bounds on the maximum number, in a graph with domination number $\gamma$, of dominating sets of size $\gamma$. We find that we can randomly generate a graph that, w.h.p., is dominated by almost all sets of size $\gamma$. At the same time, we use a modified version of the adjacency matrix to obtain lower bounds on the number of sets of a given size that do not dominate a graph on $n$ vertices.\end{abstract}
\section{Introduction}
 A set $S$ of vertices in a graph $G$ is said to dominate if each vertex in $S^C$ is adjacent to at least one vertex in $S$; the minimum cardinality $\gamma=\gamma(G)$ of such an $S$ is called the domination number.  Godbole et al \cite{gjj}  provide a construction that serves as a lower bound for the number of dominating sets of size $\gamma$ in a graph with domination number $\gamma\ge3$. They find, e.g.,  that for $\gamma=3$, a graph consisting of the union of a complete graph on $\frac{n}{3}$ vertices and a complete graph minus a perfect matching on $\frac{2n}{3}$ vertices, is dominated by no pairs, but asymptotically by $\frac{4}{9}$ of all triples. Similar results are proved for all $\gamma$.  However, we find in this paper that with a probabilistic rather than a constructive approach, we can do much better.  We prove the existence of a graph $G$ with domination number $\gamma$ that, with high probability, is dominated by almost all sets of vertices of size $\gamma$ as $n\to\infty$.
\begin{rem}
For the sake of convenience, we will often refer to a fraction that is asymptotically $(1-o(1))$ of all $\gamma$-sets as being ``all but a vanishing fraction of"  $\gamma$-sets.  Moreover, it will be tacitly assumed in such situations that the domination number of the graph $G$ in question is $\gamma$.
\end{rem}
\begin{rem}
It may be noted that our probabilistic argument utilizes relatively elementary methods. This simplicity of methods is attributable to the fact that more sophisticated arguments were either superfluous or invalid. For instance, adding or deleting an edge can dramatically change the number of dominating sets of size $\gamma$ by as many as ${n-1 \choose \gamma}$ (assuming, of course that the domination number does not change), so our quantity of interest is not finitely-Lipschitz, rendering Azuma's inequality inapplicable.
\end{rem}  

However, by dint of a result obtaining general lower bounds on the number of sets of a given size that do not dominate a graph on $n$ vertices, we find that a graph on $n$ vertices with domination number $\gamma\ge3$ is \emph{not} dominated by at least $O(n^{\gamma - 1 - \frac{1}{\gamma -1}})$ of the sets of size $\gamma$; when $\gamma=2$ it is seen in \cite{gjj} that each pair of vertices may dominate.

The question we address is only peripherally related to the maximum number of edges in a graph with given domination number \cite{vizing},\cite{sanchis},\cite{goddard}:  Even though the two answers do coincide when $\gamma=2$, in general they are not the same.  For example when $\gamma=3$, there are only $O(n^2)$ dominating sets of size 3 when the maximum number of edges are inserted; these triples are all nullified when an edge is added, causing the domination number to become 2.  Similarly, if the graph property we were investigating were ``maximum number of triangles in a graph with clique number 3," then Tur\'an's theorem tells us that we can have no more than $\sim n^2/3$ edges, which yields $\sim n^3/27$ triangles.  Hypergraph Tur\'an theory \cite{keevash}, on the other hand, tells us that a maximum of between 5/9 and $\sim0.562$ of all possible 3-edges are present in a 3-uniform hypergraph on $n$ vertices that has no tetrahedron; the correct fraction is conjectured to be 5/9.  Either of these bounds  gives us more triangles than in the extremal edge-Tur\'an case.  And of course there are fewer edges than in the edge-extreme case, if we consider the structure to be a graph, with each 3-edge contributing three 2-edges, and avoiding double counting.  Overall,  the message seems to be that problems of this type can be quite complicated.
\section{Lower Bounds on the Number of Sets that do not Dominate}

The domination number of a graph $G=(V,E)$ will be denoted by $\gamma=\gamma(G)$. Although it is often easy for small graphs to tell by inspection which sets dominate and which do not, doing so becomes difficult quickly as graphs get larger. Therefore, it will be useful to have a tool that we can use to organize a graph so that we can more easily identify and count dominating and non-dominating sets. Hence, the $s$-adjacency matrix, defined as follows:
\begin{defn}
Given a graph $G$ on $n$ vertices, an $n\, \times\, n$ matrix $\mathcal{D}_G=(d_{ij})_{i,j=1}^n$ is an \emph{$s$-adjacency matrix} of $G$ if there exists some labeling of the vertices of $G$ with the integers $1$ through $n$ such that 
$$d_{ij} = \begin{cases}
1 & i=j \text{ or vertices $i$ and $j$ are adjacent}\\
0 & \text{otherwise}
\end{cases}$$
\end{defn}
Clearly, the $s$-adjacency matrix of a graph on $n$ vertices is a symmetric $n\, \times\, n$ 0\textendash 1 matrix with ones on the diagonal, equals $A+I_n$, where $A$ is the adjacency matrix, and there is a bijection between such matrices and the set of all simple labeled graphs on $n$ vertices.
Note also that a set of vertices $a_1, a_2, \dots a_k$ in a graph $G$ is a dominating set if and only if there are no rows in $\mathcal{D}_G$ with zeroes in the $a_i$th column for all $1 \leq i \leq k$.

We now use this tool to prove our first result.

\begin{thm}
Given a graph $G$ on $n$ vertices, let $m,\gamma(G)>b\ge2$.  Then  $G$ must contain at least $g \sim \frac{n^{{(b-1)m}/{b}}}{m!}$ non-dominating sets of $m$ vertices.
\end{thm}
\begin{proof}
Take any $G$ with $\gamma(G)>b$ and $a\ge b$. Because no set of size $b$ dominates $G$, $\mathcal{D}_G$ must contain at least ${n \choose b}$ different ways of picking $b$ different $0$s in a single row. If there are at most $a$ $0$s in any row of $\mathcal{D}_G$, however, then there can be at most ${a \choose b}$ different ways of picking $b$ $0$s in any one row. Thus any one row gives rise to at most ${a\choose b}$ non-dominating sets of size $b$, and thus there are at most $n\cdot{a\choose b}$ non-dominating sets corresponding to the zeros in  the rows.  Hence, with $n$ sufficiently large so that 
\begin{equation} \label{eq:1}
 {n \choose b} >n{a \choose b}
\end{equation} holds, $\mathcal{D}_G$ cannot have at most $a$ zeros in each row, and thus must have a row with at least $a+1$ $0$s, and so $G$ has at least ${{a +1} \choose {m}}$ sets of $m$ vertices that do not dominate.   Before continuing with the proof, we examine the ramifications of the simple inequality (1) (which is impossible to satisfy for $b=1$) and of its reverse.
\begin{lem}
Given any $b$, for sufficiently large $a$, $n{a \choose b} \geq {n \choose b} \Rightarrow n \frac{a^b}{b!} \geq \frac{n^b}{b!}$.
\end{lem}
\begin{proof}
If $a \geq n$, then the result is trivial. Otherwise, given some fixed $b$, consider the ratio
$
{{a \choose b}}/{\frac{a^b}{b!}}.$

This expression simplifies to $$1-c_1a^{-1}+c_2a^{-2}-c_3a^{-3}+ \dots \pm c_{b-1}a^{-(b-1)}.$$
where for $1\le i\le b-1$, $c_i$ represent the absolute values of the coefficients of $a^{b-i}$  in the expansion of 
$${a \choose b}{b!}=\Pi_{i=0}^{b-1}(a-i).$$
Note that for a fixed $b$, these coefficients do not vary with $a$. Since $
{{a \choose b}}/{\frac{a^b}{b!}}$ is positive and bounded above by $1$, by the above it is of the form $1-\epsilon(a)$ with $$0<\epsilon(a)=c_1a^{-1}-c_2a^{-2}+c_3a^{-3}+ \dots \mp c_{b-1}a^{-(b-1)}<1.$$
Differentiating, we see that for sufficiently large $a$, the first term of $\epsilon'(a)$ dominates, and so $\epsilon(a)$ is decreasing for $a$ sufficiently large, and hence $
{{a \choose b}}/{\frac{a^b}{b!}}$ is monotone increasing if $a\ge a_0$. Therefore, for $n>a\ge a_0; n{a \choose b}\ge{n\choose b}$ implies that $(1-\epsilon(a))n\frac{a^b}{b!} \geq (1-\epsilon(n))(\frac{n^b}{b!})$ with $\epsilon(n)<\epsilon(a)$, so $$n\frac{a^b}{b!}\geq \frac{1-\epsilon(n)}{1-\epsilon(a)}\cdot\frac{n^b}{b!}>\frac{n^b}{b!}.$$  This proves the lemma.\hfill
\end{proof}
Now fix a small $b$ and select $a'$ sufficiently large so as to satisfy the hypothesis of Lemma 2.3, i.e. if $n{a\choose b}$ were to be larger than ${n\choose b}$, then we would have $n{a^b}/{b!} \geq {n^b}/{b!}$  Next choose $n$ sufficiently large  so as to satisfy ``the opposite inequality" (1) for that value of $a'$, and let $a\geq a'$ be the largest integer satisfying (1) for this value of $n$. Then there must be a row of $\mathcal{D}_G$ with $a + 1$ $0$s, but we also have $\frac{n^{b}}{b!} \leq n \frac{(a+1)^{b}}{b!}$, and hence $a + 1 \geq n^\frac{b-1}{b}$. Therefore, $G$ must have at least $g={n^\frac{b-1}{b} \choose m}\sim \frac{n^\frac{(b-1)m}{b}}{m!}$ $m$-sets that do not dominate.  This proves Theorem 2.2.
\end{proof}
The following corollary follows directly from substituting $m=\gamma$ and $b=\gamma-1$:
\begin{cor}
For $m=\gamma$, $b=\gamma-1$, $g \sim \frac{n^{\gamma-1 -\frac{1}{\gamma-1}}}{\gamma!}$, i.e., in any graph with domination number $\gamma$, there are at least $\Omega\lr {n^{\gamma-1 -({1}/{(\gamma-1)})}}\rd$ non-dominating sets of size $\gamma$.
\end{cor}

\section{Domination by a large fraction of $\gamma$-sets}

As we have seen, a graph with domination number $\gamma$ is \emph{not} dominated by a relatively large number of sets of a size $\gamma$. However, the lower bound on non-dominating sets represents a small fraction of the total number of sets of size $\gamma$, and, as we will see in this section, we can randomly construct a graph that has high probability of having domination number $\gamma$ and of being dominated by all but an arbitrarily small fraction of its $\gamma$-sets.
\begin{thm}
 Given any $\gamma \in \nn$, there exists a sequence $\epsilon_{\gamma, n}$, such that, w.h.p., the Erd\H os-R\'enyi random graph $G(n,1-\epsilon_{\gamma,n})$  has domination number $\gamma$ and is dominated by all but a vanishing fraction of the  $\gamma$-sets ($n\to\infty$).
\end{thm}

\begin{proof}
Let $X_{\gamma-1}$ be the number of dominating sets of size one less than the putative domination number.  By Markov's inequality and linearity of expectation,
\begin{eqnarray}\p(X_{\gamma-1}\ge1)&\le&\e(X_{\gamma-1})
\le{{n}\choose{\gamma-1}}(1-\epsilon_{\gamma,n}^{\gamma-1})^{n-\gamma+1}\nonumber\\
&\leq& \frac{n^{\gamma-1}}{(\gamma-1)!}\exp\{-(n-\gamma+1)\epsilon_{\gamma,n}^{\gamma-1}\}\to0
\end{eqnarray}
provided $\epsilon_{\gamma,n}\ge\sqrt[\gamma-1]{\frac{(\gamma-1+\delta)\log n}{n}}$, where $\delta>0$ is arbitrary.  So for such $\epsilon_{\gamma,n}$, there will with high probability exist no dominating sets of size $\gamma-1$.  By a similar argument, and using the inequality $1-u\ge\exp\{-u/(1-u)\}$, we see that \[\e(X_\gamma)={n \choose \gamma}(1-\epsilon_{\gamma,n}^\gamma)^{n-\gamma}\ge{n \choose \gamma}\exp\lc-\frac{(n-\gamma)\epsilon_{\gamma,n}^\gamma}{1-\epsilon_{\gamma,n}^\gamma}\rc.\]
As $n\to\infty$, $\e(X_\gamma)/{n\choose \gamma}\to1$ provided $\epsilon_{\gamma,n}\ll\frac{1}{\sqrt[\gamma]{n}}$. For such $\epsilon_{\gamma,n}$, the expected number of dominating sets of size $\gamma$ is asymptotic to the total number of sets of size $\gamma$.
The two ranges of desired values for $\epsilon_{\gamma,n}$ are not disjoint; for example, $\epsilon_{\gamma,n}=\frac{\log n}{\sqrt[\gamma-1]{n}}$ lies in their intersection. For such values of $\epsilon_{\gamma,n}$, we almost surely have no dominating sets of size $\gamma-1$, and the expected number of sets of size $\gamma$ is all but a vanishing fraction of all of the sets of size $\gamma$. Assume next that $\e(X_\gamma)={n\choose\gamma}(1-o(1))$ but that $X_\gamma\le q{n\choose\gamma}$ with probability $p>0$ for some $q<1$.  Then $\e(X_\gamma)\le pq{n\choose\gamma}+(1-p){n\choose\gamma}=(1-p+pq){n\choose\gamma}\ne{n\choose\gamma}(1+o(1))$, a contradiction.
Thus, for $n$ large, it is impossible for the value of $X_\gamma$ to fall further than a nonvanishing fraction below its expected value. And so this process, with high probability, constructs a graph of domination number $\gamma$ in which all but a vanishing fraction of sets of size $\gamma$ dominate.  However, we want to quantify the above behavior in terms of a concentration result that mirrors Corollary 2.4.  Accordingly, we compute the variance of $X_\gamma$ and showing that it grows more slowly than does $E^2X_\gamma$.   We have
$$X_\gamma=\displaystyle\sum_{i=1}^{{n \choose \gamma}} I_i,$$
where $I_i$ equals 1 if the $i$th set of $\gamma$  vertices (also denoted by $i$) is a dominating set and 0 otherwise.   Thus
$$E(X_\gamma^2)=\displaystyle\sum_{r=0}^{\gamma} \displaystyle\sum_{|i\cap j|=r}^{} E(I_i I_j),$$
since the probability that two $\gamma$-sets both dominate depends only on how many points they share. 
%Furthermore, since only on the order of $n^{2\gamma-r}$ pairs of $\gamma$-sets with intersection size $r$ exist, means that only 
Consider first the $r=0$ term, 
%of order greater than $n^{2\gamma-1}$, so we may focus only on pairs of sets with empty intersection as the other cases only produce lower order terms.
which is the product of three terms: the number of pairs of $\gamma$-sets with empty intersection; the probability that two non-overlapping $\gamma$-sets dominate each other; and the probability that two non-overlapping $\gamma$-sets both dominate the other $n-2\gamma$ points.
The number of such pairs equals
${n \choose 2\gamma}{{2\gamma \choose \gamma}}/{2}=\frac{n^{2\gamma}}{2(\gamma !)^2}+O(n^{2\gamma-1}).$
The probability that two such sets dominate each other is found by performing inclusion-exclusion on the events that a specific point is not dominated by the other set. Each of the $2\gamma$ points involved has a $\epsilon_{\gamma,n}^\gamma$ probability of not being dominated, as all $\gamma$ potential edges between it and the other set must be missing. The events where we require multiple points to not be dominated are less likely as more and more edges need to be missing. The highest order terms here are:
$$1-2\gamma \epsilon_{\gamma,n}^\gamma +\gamma^2 \epsilon_{\gamma,n}^{2\gamma-1} +(\gamma^2-\gamma)\epsilon_{\gamma,n}^{2\gamma}$$
Since $\epsilon_{\gamma,n}\ll\frac{1}{\sqrt[\gamma]{n}}$, the largest term above, other than 1, is $o({1}/{n})$. (Incidentally, the separate $\epsilon_{\gamma,n}^{2\gamma-1}$ and $\epsilon_{\gamma,n}^{2\gamma}$ terms appear because the probability that two points are both not dominated by the other set depends on whether or not they are members of the same set.)

The third factor, the probability that both of our $\gamma$-sets dominate the remaining $n-2\gamma$ points, equals $(1-\epsilon_{\gamma,n}^\gamma)^{2n-4\gamma}=1-o(1)$.  Thus the net contribution of the $r=0$ case to the second moment is
$$\lr\frac{n^{2\gamma}}{2(\gamma !)^2}+O(n^{2\gamma-1})\rd(1-2\gamma \epsilon_{\gamma,n}^\gamma  + O(\epsilon_{\gamma,n}^{2\gamma-1}))(1-(2n-4\gamma)\epsilon_{\gamma,n}^\gamma+O(n^2\epsilon_{\gamma,n}^{2\gamma})),$$  which simplifies to $$\frac{n^{2\gamma}}{2(\gamma !)^2}(1-\epsilon_{\gamma,n}^\gamma)^{2n-6\gamma}(1+o(1))=\frac{n^{2\gamma}}{2(\gamma !)^2}(1-\epsilon_{\gamma,n}^\gamma)^{2n}(1+o(1)).$$
As the $r\geq 1$ cases yield no terms of higher order than $n^{2\gamma-1}(1-\epsilon_{\gamma,n}^\gamma)^{2n}(1+o(1))$, the above are  the only high-order terms of $E(X_\gamma^2)$.

Note next that \begin{eqnarray*}(EX_\gamma)^2&=&\lr{n \choose \gamma}(1-\epsilon_{\gamma,n}^\gamma)^{n-\gamma})\rd^2\\&=&\frac{n^{2\gamma}}{2(\gamma !)^2}(1-\epsilon_{\gamma,n}^\gamma)^{2n-2\gamma}+O(n^{2\gamma-1})\\&=&\frac{n^{2\gamma}}{2(\gamma !)^2}(1-\epsilon_{\gamma,n}^\gamma)^{2n}(1+o(1)).\end{eqnarray*} 
Consequently,
\[{\rm Var}(X_\gamma)=E(X_\gamma^2)-\e^2(X_\gamma)=O(n^{2\gamma-1}),\]
and by Chebychev's inequality, for any $\phi(n)\to\infty; \phi(n)=o(n^{1/2})$,
\[\p(|X_\gamma-\e(X_\gamma)|\ge\phi(n)n^{\gamma-1/2})\le\frac{A}{\phi^2(n)}\]
for aome $A>0$.  It follows, taking $\phi(n)=\log n$ and $\epsilon_{\gamma,n}=\log n/n^{1/(\gamma-1)}$,  that w.h.p. for $\gamma\ge3$, \begin{eqnarray*}X_\gamma&\ge&\e(X_\gamma)-\frac{\log n}{\sqrt n}n^\gamma\\&=&{n\choose\gamma}-{n\choose\gamma}\frac{\log^\gamma n}{n^{1/(\gamma-1)}}-\frac{\log n}{\sqrt n}n^\gamma\\&=&{n\choose\gamma}-O(\log^\gamma n\cdot n^{\gamma-\frac{1}{(\gamma-1)}}).\end{eqnarray*} This completes the proof.\hfill\end{proof}
Together with Corollary 2.4, we have also proved 
\begin{cor}  The maximum number $M=M_{n,\gamma}$ of dominating sets of size $\gamma$ in a graph $G$ on $n$ vertices and with domination number $\gamma\ge3$ satisfies $${n\choose\gamma}-O(\log^\gamma n\cdot n^{\gamma-\frac{1}{(\gamma-1)}})\le M\le {n\choose \gamma}-\Omega( {n^{\gamma-1 -({1}/{(\gamma-1)})}}).$$
\end{cor}

\section{Open Questions}
\label{section5}
First, we believe that the bounds in our results can be tightened, particularly by increasing the lower bound on number of non-dominating sets. This intuition arises from the fact that the lower bound is arrived at by considering the number of non-dominating sets, the existence of which are implied by the 0s in a single row of $\mathcal{D}_G$.  It is likely, however, that there are many more non-overlapping $\gamma$-sets of 0s among the $n-1$ other rows.  Second, as asked by Godbole et al \cite{gjj}, can we do better than in that paper by constructive means?

\medskip

\noindent {\bf Acnowledgments}  The research of all four authors was supported by NSF Grant 1263009.

% ********************

\end{document}